\documentclass[a4paper,11pt]{article}
\usepackage[T1]{fontenc}
\usepackage{lmodern}
\usepackage[margin=29mm]{geometry}
\usepackage{amsmath,amssymb,amsthm,microtype}
\usepackage[hidelinks]{hyperref}
\usepackage{color}

\hypersetup{pdftitle={Flocking with finite-speed communication: a Doppler-inspired Lyapunov approach},pdfsubject={Flocking with finite-speed communication: a Doppler-inspired Lyapunov approach}}
\numberwithin{equation}{section}
\newtheorem{theorem}{Theorem}
\newtheorem{definition}{Definition}
\newtheorem{proposition}[theorem]{Proposition}
\newtheorem{lemma}[theorem]{Lemma}
\newtheorem{corollary}[theorem]{Corollary}
\theoremstyle{remark}
\newtheorem{remark}[theorem]{Remark}

\newcommand{\R}{\mathbb R}
\newcommand{\N}{\mathbb N}

\newcommand{\cvel}{\mathfrak c}
\newcommand{\svel}{\mathfrak s}

\newcommand{\HH}{\mathcal H}
\newcommand{\KK}{\mathcal K}
\newcommand{\LL}{\mathcal L}
\newcommand{\JJ}{J}
\newcommand{\pp}{\psi\left(  \cvel\overline\tau  \right)}
\renewcommand{\d}{\mathrm d}
\DeclareMathOperator{\diam}{diam}
\DeclareMathOperator{\co}{co}
\DeclareMathOperator{\sgn}{sgn}
\allowdisplaybreaks[1]

\def\({\begin{eqnarray}}
\def\){\end{eqnarray}}
\def\[{\begin{eqnarray*}}
\def\]{\end{eqnarray*}}

\def\part#1#2{\frac{\partial #1}{\partial #2}}

\def\d{\mathrm{d}}

\def\ta#1#2{\tau_{{#1}{#2}}}
\def\taij{\ta{i}{j}}

\def\solx{\mathbf{x}}
\def\solv{\mathbf{v}}

\def\dx{{d_{\solx}}}
\def\dv{{d_{\solv}}}

\title{Flocking with finite communication speed:\\ a Doppler-inspired Lyapunov approach}

\author{J. Haskovec\footnote{Mathematical and Computer Sciences and Engineering Division, King Abdullah University of Science and Technology, Thuwal 23955-6900, Kingdom of Saudi Arabia. Email: {\it jan.haskovec@kaust.edu.sa}}, M. Rodriguez Cartabia\footnote{IMAS (UBA-CONICET) and Departamento de Matem\'atica, Facultad de Ciencias Exactas y Naturales, Universidad de Buenos Aires, Ciudad Universitaria, 1428 Buenos Aires, Argentina. Email: {\it mrodriguezcartabia@dm.uba.ar}}}

\date{}

\begin{document}
\maketitle

\begin{abstract}
We study the Cucker--Smale model with state-dependent transmission delays
arising from finite-speed information propagation.
Our main contribution is a Lyapunov functional inspired by the Doppler effect,
combining a delay potential with a logarithmic correction involving the
receiver velocity. Its dissipation yields uniform delay bounds without
relying on velocity contraction estimates.
For two agents in one spatial dimension, we prove asymptotic flocking for
arbitrary compatible subluminal initial histories, assuming only
that the influence function has a divergent integral at infinity.
A bounded geometric correction extends this result to arbitrary spatial
dimension when all values of the initial velocity histories lie on one common
affine line.
For multiple agents in one dimension, we obtain uniform delay bounds, and thus flocking,
provided that a fixed agent is leftmost for all sufficiently large times.
In all these flocking regimes, the spatial diameter remains uniformly bounded
and the velocity diameter decays exponentially.
We also derive a conditional criterion involving logarithmic sums when the
leftmost agent changes infinitely often.
Finally, we construct a multidimensional analogue of the Doppler-inspired functional and
identify an additional assumption needed to obtain delay bounds in the general
case.
\end{abstract}
\vspace{2mm}

\vspace{2mm}



\section{The model}\label{sec:model}

The Cucker--Smale \cite{CS1, CS2} model with finite speed of information propagation, introduced and studied in \cite{H:KRM},
describes a group of $N\geq 2$ autonomous agents located in the physical space $\R^d$, $d\geq 1$.
Each agent $i\in[N]:=\{1,\ldots,N\}$ is characterized by its phase-space coordinates $(x_i(t), v_i(t))\in\R^{2d}$,
where $x_i(t)$ denotes the position and $v_i(t)$ the velocity of the $i$-th agent.
In the sequel we shall denote $\solx := (x_1, \ldots, x_N)\in \R^{Nd}$ 
and $\solv := (v_1, \ldots, v_N)\in \R^{Nd}$. 
The system is subject to the following collective dynamics,
\begin{align}
   \dot x_i(t) &= v_i(t),   \label{eq:CS1} \\
    \dot v_i(t) &=  \sum_{\substack{j\in[N]\\j\neq i}} \psi( | x_j(r_{ij}(t)) - x_i(t)| )  ( v_j(r_{ij}(t)) - v_i(t)),   \label{eq:CS2}
\end{align}
with the retarded times $r_{ij}(t) := t-\tau_{ij}(t)$, and the state-dependent
transmission delays $\tau_{ij}=\tau_{ij}(t)$ satisfy the light-cone relation
\(   \label{def:tauij}
    \cvel\tau_{ij}(t) = \left| x_i(t)-x_j(t-\tau_{ij}(t)) \right|.
\)
Here $\cvel>0$ is the prescribed speed of information propagation.
Since we do not consider the mean-field limit of the system here, we omit the standard normalization factor of
$\frac{1}{N-1}$ on the right-hand side of \eqref{eq:CS2}, which was present in \cite{H:KRM}.
This significantly simplifies the notation in the paper.

The positive influence function $\psi:[0,\infty)\to(0,\infty)$ is assumed to be bounded, uniformly Lipschitz continuous and nonincreasing.
Without loss of generality we may assume that $\psi(0)=1$.
Indeed, denoting $M:=\psi(0)>0$, the space-time rescaled variables
\begin{equation*}
\widehat x_i(s):=M x_i(s/M),\qquad
\widehat v_i(s):=v_i(s/M),\qquad
\widehat\tau_{ij}(s):=M\tau_{ij}(s/M),\qquad
\widehat\psi(r):=\frac{1}{M}\psi(r/M),
\end{equation*}
satisfy \eqref{eq:CS1}--\eqref{def:tauij} and leave the propagation speed $\mathfrak c$ unchanged.
We therefore impose $\psi(0)=1$ in the sequel.

System \eqref{eq:CS1}--\eqref{eq:CS2} is equipped with the initial datum
\(  \label{IC}
    x_i(t) = x_i^0(t), \quad v_i(t) = v_i^0(t) \qquad\mbox{for } i\in[N],\quad t\leq 0,
\)
where $x_i^0=x_i^0(t)$, $v_i^0=v_i^0(t)$ are continuous paths on $(-\infty,0]$.
We impose the physically relevant assumption that the initial velocities are subluminal, i.e.,
\( \label{cond:sublum}
    \svel := \max_{i\in[N]} \sup_{t\leq 0} |v_i^0(t)| <\cvel.
\)
Moreover, we impose the compatibility condition
\(   \label{IC:comp}
   x^0_i(t) = x^0_i(0) + \int_0^t v^0_i(s) \d s \qquad\mbox{for } i\in[N],\quad t\leq 0.
\)
Then, as proved in \cite[Theorem 1]{H:KRM}, the values of the initial datum \eqref{IC}
are only relevant on the bounded time interval $[-S_0,0]$ with $S_0:= \frac{\dx(0)}{\cvel-\svel}$
and $\dx(0) := \max_{i,j \in[N]}|x_i^0(0) - x_j^0(0)|$.
Therefore, we only need to prescribe $x_i^0(0)$ and $v_i^0=v_i^0(t)$ for $t\in [-S_0,0]$.

The following result on global existence of unique solutions of the system \eqref{eq:CS1}--\eqref{IC:comp}
was established in \cite[Theorem 1]{H:KRM}. We also refer to \cite{HKWW} for the general theory
of functional differential equations with state-dependent delays.

\begin{proposition}[Existence and uniqueness of global solutions]  \label{thm:ex}
Under the assumptions stated above, system \eqref{eq:CS1}--\eqref{IC} admits a unique global solution 
$\solx\in C([-S_0,\infty);\R^{Nd})$, $\solv \in C([-S_0,\infty);\R^{Nd})$ such that
\(  \label{eq:subluminal}
    |v_i(t)| \leq \svel \quad \mbox{ for all } t \geq 0,\, i\in [N],
\)
with $0 \leq \svel < \cvel$ given by \eqref{cond:sublum}.
\end{proposition}

Let us point out that the subluminal condition \eqref{cond:sublum} imposed on the initial
velocity trajectories is propagated along the solution \eqref{eq:subluminal}.
This ensures the unique solvability of the light-cone relation \eqref{def:tauij} for all $t\geq 0$.
Moreover, it implies global $\svel$-Lipschitz continuity of the trajectories $x_i(t)$, $i\in[N]$.

This paper is concerned with a proof of asymptotic flocking, which is defined as follows.

\begin{definition}[Asymptotic flocking]\label{def:flocking}
We say that a solution 
$(\solx(t), \solv(t))_{t\geq 0}$ of the system \eqref{eq:CS1}--\eqref{IC} exhibits \emph{(asymptotic) flocking} if
\[  
   \sup_{t\geq 0} \dx(t) < \infty,\qquad \lim_{t\to\infty} \dv(t) = 0,
\]
where the spatial and, resp., velocity diameters are defined as
\[  \label{dXdV}
   \dx(t) := \max_{i,j \in[N]}|x_i(t) - x_j(t)| \quad \mbox{and} \quad \dv(t) := \max_{i,j \in[N]}|v_i(t) - v_j(t)|.
\]
\end{definition}

We refer to \cite{ChoiSeo2023, Continelli2023, H:SIADS, H:PAMS, HRC:PAMS, PT:CMS} for flocking results in related Cucker-Smale systems with delay.

\section{Main result and its discussion}

The main result of this paper is the following.

\begin{theorem}\label{thm:main}
Assume that the initial histories satisfy the subluminality and compatibility conditions \eqref{cond:sublum}
and \eqref{IC:comp}, and that
\begin{equation}\label{cond:flocking}
    \int_0^\infty \psi(s)\,\d s=\infty.
\end{equation}
Then the unique global solution of
\eqref{eq:CS1}--\eqref{IC} exhibits asymptotic flocking in the sense of
Definition~\ref{def:flocking} in each of the following cases:
\begin{enumerate}
    \item[(i)] $N=2$ and $d=1$.
    \item[(ii)] $N=2$ and $d\geq2$, and the initial velocity histories $v_1^0$, $v_2^0$ lie on one common affine line. 
    \item[(iii)] $N>2$ and $d=1$, and a fixed agent is leftmost for all sufficiently large times.
    \end{enumerate}
In all these cases the transmission delays $\tau_{ij}=\tau_{ij}(t)$ and the spatial diameter $\dx=\dx(t)$
are uniformly bounded,
and the velocity diameter $\dv=\dv(t)$ converges to zero exponentially as $t\to\infty$.
\end{theorem}

The fat-tail condition \eqref{cond:flocking} is known to be optimal
for unconditional flocking for the original Cucker-Smale system without delay \cite{CS1, CS2, Ha-Liu}.
Therefore, Theorem \ref{thm:main}(i) establishes flocking for $N=2$ and $d=1$ under minimal assumptions.
Part (ii) then generalizes the result for $d\geq 2$, however, the assumption that the initial velocity histories
lie on one common affine line allows to effectively project the system's dynamics into one dimension.
A truly spatially multidimensional setting introduces technical difficulties that we were not able to resolve.

Part (iii) of Theorem \ref{thm:main} extends the one-dimensional result to $N>2$,
under the additional assumption that a fixed agent is leftmost for all sufficiently large times;
obviously, the same result also holds if a fixed agent is eventually rightmost.
We conjecture that this assumption holds for all solutions of \eqref{eq:CS1}--\eqref{IC}
satisfying \eqref{cond:sublum} and \eqref{IC:comp}; however, we were not able to provide a proof.

The main novelty of our analysis is the introduction of a Lyapunov
functional inspired by the Doppler effect.
For $N=2$ it consists of a delay potential defined in terms of the primitive of the influence function $\psi$,
and a logarithm of the so-called Doppler receiver factor.
This factor is the Doppler frequency ratio associated with the retarded signal
received by the respective agent.
A dissipation identity, combined with the fat tail assumption \eqref{cond:flocking},
then directly implies a uniform bound on the delays,
without using any velocity contraction estimates.
For $N>2$ we introduce a generalization of the Doppler-inspired
functional, and under the additional assumption that a fixed agent
is leftmost for all sufficiently large times,
its decay again implies a uniform bound on the delays.

Once the uniform bound on the delays is established,
a uniform bound on the spatial diameter $\dx=\dx(t)$
immediately follows. Moreover, exponential decay of
the velocity diameter $\dv=\dv(t)$ is then proved directly.
No bootstrap or smallness assumptions on the initial diameters are needed,
and the propagation speed need only satisfy $\cvel>\svel$.
In contrast, the approach of \cite{H:KRM} applies to arbitrary $N$ and $d$,
but uses a bootstrap argument
linking velocity decay, spatial dispersion, and transmission delays.
Its restrictive closure conditions couple the influence function,
initial fluctuations, and propagation speed; for fixed $\cvel$,
they require small position and velocity fluctuations.

\subsection*{Organization of the paper}
The paper is organized around the separation between delay control
and velocity alignment. In Section~\ref{sec:hull} we collect results
that hold for any $N\geq 2$ and $d\geq 1$.
We first establish causality, which amounts to strict monotonicity
of all retarded-time maps $t \mapsto r_{ij}(t)$.
We then prove nonexpansivity of the convex hull of the relevant
velocity histories. Both properties hold without an a priori bound
on the delays. Finally, under the assumption of uniformly bounded delays,
we obtain uniform contraction of the velocity-history convex hull over fixed
time intervals, and hence exponential decay of the velocity diameter.
Since bounded delays also bound the spatial diameter, these results
reduce the flocking problem to establishing a uniform delay bound.

Section~\ref{sec:1D} carries out this step for two agents in one
dimension using the Doppler-inspired functional described above.
Its decay follows from the fact that for $N=2$ the
receiver's acceleration is generated by a single emitter,
and directly implies a uniform bound on the delays.
This proves part (i) of Theorem \ref{thm:main}.

Section~\ref{sec:multid-flocking} extends the construction to two
agents in multiple spatial dimensions with initial velocity histories
lying on one common affine line. 
Nonexpansivity confines both velocities to the convex hull of all relevant initial velocity values,
so that the dynamics can be projected onto a fixed
velocity direction. A bounded geometric correction to the logarithmic
term accounts for the transverse separation and yields
a modified Lyapunov functional with the same delay-control mechanism.
Thus part (ii) of Theorem \ref{thm:main} is established.

Section~\ref{sec:N1d} examines the extension to multiple agents in one dimension.
The main idea is that when the receiver is leftmost, all incoming signals
have the same orientation. Summing their delay potentials then
leads to a generalization of the Doppler-inspired functional.
If an agent remains leftmost after some finite time, its decay gives a uniform bound on all delays.
This proves part (iii) of Theorem \ref{thm:main}.

In the rest of Section~\ref{sec:N1d} we also discuss the case when
the leftmost agent changes infinitely many times (although we hypothesise that this case never occurs).
Here we employ the decay of the Doppler-inspired functional on each interval
separating the switches of the leftmost agent.
Concatenating the interval estimates produces a sum of logarithmic terms stemming from the switches.
Uniform boundedness of the delays then follows provided that the partial logarithmic sums are uniformly bounded above.

Finally, Section~\ref{sec:optical-action} presents a multidimensional
analogue of the Doppler-inspired functional, which satisfies a
dissipation identity for arbitrary $N\geq 2$ and $d\geq 1$.
Unfortunately, a uniform lower bound on the functional
is the missing ingredient for delay control,
leaving the general multidimensional flocking problem a subject for future research.

\section{Results valid for general $N\geq 2$ and $d\geq 1$}\label{sec:hull}

In this section we prove three results that hold for all solutions of the system \eqref{eq:CS1}--\eqref{IC}
with any $N\geq 2$ and $d\geq 1$, assuming only that the initial datum satisfies \eqref{cond:sublum} and \eqref{IC:comp}.
In particular, we do not assume validity of the fat-tail assumption \eqref{cond:flocking} here.

We recall that for $i,j\in[N]$ and $t\geq 0$ we denote $r_{ij}(t):=t-\taij(t)$ the retarded time.
We say that the system satisfies the \emph{causality condition} if every map $t \mapsto r_{ij}(t)$ is strictly increasing.
The next lemma establishes causality, even in a stronger version, for the solutions of \eqref{eq:CS1}--\eqref{eq:CS2}.

\begin{lemma}\label{lem:causality}
For any $1 \leq \lambda < \frac{\cvel+\svel}{2\svel}$, the functions $t \mapsto t- \lambda\tau_{ij}(t)$ are strictly increasing for all $t\geq 0$ and $i,j\in[N]$,
along the solutions of \eqref{eq:CS1}--\eqref{cond:sublum}.
In particular, the causality condition is satisfied globally.
\end{lemma}

\begin{proof}
Fix $\lambda \in \left[ 1, \frac{\cvel+\svel}{2\svel} \right)$.
For contradiction, suppose there are $0 \leq t_1<t_2$ such that
$$t_1- \lambda\tau_{ij}(t_1) \geq t_2- \lambda\tau_{ij}(t_2).$$
Note that this implies 
$\lambda(\tau_{ij}(t_2) - \tau_{ij}(t_1)) \geq t_2-t_1 >0$,
and thus $\tau_{ij}(t_2) - \tau_{ij}(t_1) >0$. 
From \eqref{def:tauij} and \eqref{eq:subluminal} we then have
\begin{align*}
 \cvel \left( \tau_{ij}(t_2)-\tau_{ij}(t_1) \right) & =   \left|x_i(t_2)-x_j(t_2-\tau_{ij}(t_2))\right| -  \left|x_i(t_1)-x_j(t_1-\tau_{ij}(t_1))\right|   \\
  &\leq  \left|x_i(t_2)-x_j(t_2-\tau_{ij}(t_2))-x_i(t_1)+x_j(t_1-\tau_{ij}(t_1))\right|   \\
  &\leq  |x_i(t_2)-x_i(t_1)| +  |x_j(t_1-\tau_{ij}(t_1))-x_j(t_2-\tau_{ij}(t_2)) |   \\
  &\leq \svel \left[ |t_2-t_1| + |t_1-\tau_{ij}(t_1)-t_2+\tau_{ij}(t_2)| \right]   \\
  &\leq \svel \left[ |t_2-t_1| + |t_1-\lambda\tau_{ij}(t_1)-t_2+\lambda\tau_{ij}(t_2)|  + |(\lambda-1)(\tau_{ij}(t_1) - \tau_{ij}(t_2)) |  \right]    \\
  &= \svel \left[ t_2-t_1+ t_1-\lambda\tau_{ij}(t_1)-t_2+\lambda\tau_{ij}(t_2) + (\lambda-1)(\tau_{ij}(t_2) - \tau_{ij}(t_1))   \right]    \\
  & = \svel (2\lambda-1) \left( \tau_{ij}(t_2)-\tau_{ij}(t_1)  \right).
\end{align*}
This implies $\cvel^{-1}\svel (2\lambda-1) \geq 1$,
which is a contradiction to $\lambda \in \left[ 1, \frac{\cvel+\svel}{2\svel} \right)$.
Consequently, $t \mapsto t-\lambda\tau_{ij}(t)$ is strictly increasing in $t$.
Causality follows with $\lambda:=1$.
\end{proof}

\begin{corollary}\label{cor:retarded-increments}
Denote
\begin{equation*}
   \kappa:=\frac{\cvel-\svel}{\cvel+\svel} \in (0,1].
\end{equation*}
Then, for every $0\leq t_1 \leq t_2$,
\begin{equation}\label{eq:retarded-increments}
    \kappa(t_2-t_1)\leq r_{ij}(t_2)-r_{ij}(t_1) \leq\kappa^{-1}(t_2-t_1).
\end{equation}
Moreover,
\begin{equation}\label{eq:delay-distance}
 \frac{|x_i(t)-x_j(t)|}{\cvel+\svel}\leq\tau_{ij}(t)
 \leq\frac{|x_i(t)-x_j(t)|}{\cvel-\svel}.
\end{equation}
\end{corollary}

\begin{proof}
If $\svel=0$, the position trajectories and the delays are constant,
and the assertion is immediate. Suppose $\svel>0$ and write
\begin{equation*}
h:=t_2-t_1>0,\qquad R:=r_{ij}(t_2)-r_{ij}(t_1).
\end{equation*}
By Lemma~\ref{lem:causality}, for every $1\leq\lambda< \frac{\cvel+\svel}{2\svel}$,
\begin{equation*}
0<h-\lambda\bigl(\tau_{ij}(t_2)-\tau_{ij}(t_1)\bigr)
=(1-\lambda)h+\lambda R.
\end{equation*}
Consequently, $R>(1-\lambda^{-1})h$. Letting $\lambda\to \frac{\cvel+\svel}{2\svel}$ gives
\begin{equation*}
R\geq\frac{\cvel-\svel}{\cvel+\svel}h=\kappa h,
\end{equation*}
which is the lower bound in \eqref{eq:retarded-increments}.

To obtain the upper bound, use \eqref{def:tauij}, the reverse
triangle inequality, and the $\svel$-Lipschitz continuity of the
position trajectories,
\begin{align*}
\cvel(R-h)
&=\cvel\bigl(\tau_{ij}(t_1)-\tau_{ij}(t_2)\bigr)\\
&\leq |x_i(t_2)-x_i(t_1)|
  +|x_j(r_{ij}(t_2))-x_j(r_{ij}(t_1))|\\
&\leq\svel(h+R).
\end{align*}
Here we used $R>0$ in the last inequality. Rearranging yields
$(\cvel-\svel)R\leq(\cvel+\svel)h$, which proves the upper bound
in \eqref{eq:retarded-increments}.

For \eqref{eq:delay-distance}, use the triangle inequality in \eqref{def:tauij},
\begin{align*}
  \cvel\tau_{ij} &=|x_i(t)-x_j(t-\tau_{ij})|\\
    &\leq |x_i(t)-x_j(t)| + |x_j(t)-x_j(t-\tau_{ij})| \\
    & \leq  |x_i(t)-x_j(t)| + \svel\tau_{ij}.
\end{align*}
Conversely,
\begin{align*}
   |x_i(t)-x_j(t)| &\leq |x_i(t)-x_j(t-\tau_{ij})|+|x_j(t-\tau_{ij})-x_j(t)|\\
   &\leq\cvel\tau_{ij} + \svel\tau_{ij}.
\end{align*}
Thus $(\cvel-\svel)\tau_{ij} \leq |x_i(t)-x_j(t)| \leq(\cvel+\svel)\tau_{ij}$,
and \eqref{eq:delay-distance} follows since $\cvel>\svel$.
\end{proof}

Let us note that the lower estimate in \eqref{eq:delay-distance} gives
\[
    \dx(t)\leq(\cvel+\svel) \max_{i,j \in [N]} \tau_{ij}(t).
\]
Consequently, a uniform bound on the delays immediately gives a uniform bound on the spatial diameter $\dx$.

In the sequel we will use the following auxiliary result.

\begin{lemma}\label{lem:convex-integral}
Let $\mathcal K \subset \mathbb{R}^d$ be a nonempty, closed and convex set, and let $a<b$. Let $w:[a,b]\to[0,\infty)$ and $v:[a,b]\to\mathbb{R}^d$ be continuous
functions such that $v(t)\in\mathcal K$ for all $t\in[a,b]$. Denote $m := \int_a^b w(t)\,\d t .$ Then
\[
   \int_a^b w(t)\,v(t)\,\d t \in m\,\mathcal K.
\]
\end{lemma}

\begin{proof}
If $m=0$, then $w\equiv 0$ and there is nothing to prove. Assume therefore $m>0$.

For $n\in\mathbb{N}$, set $\Delta_n := (b-a)/n$ and $t_l := a + l\Delta_n$, $l=0,\dots,n-1$,
and define the Riemann sums
\[
   S_n := \sum_{l=0}^{n-1} \Delta_n\, w(t_l)\, v(t_l),
   \qquad
   M_n := \sum_{l=0}^{n-1} \Delta_n\, w(t_l),
\]
and
\[
   \lim_{n\to\infty} S_n = \int_a^b w(t)\,v(t)\,\d t,
   \qquad
   \lim_{n\to\infty} M_n = m > 0 .
\]
In particular, there exists $n_0\in\mathbb{N}$ such that $M_n>0$ for all $n\ge n_0$.
For such $n$ we write
\[
   \frac{S_n}{M_n} = \sum_{l=0}^{n-1} \lambda^{(n)}_{l}\, v(t_l),
   \qquad
   \lambda^{(n)}_{l} := \frac{\Delta_n\, w(t_l)}{M_n}.
\]
The coefficients $\lambda^{(n)}_{l}$ are nonnegative and, by the definition of $M_n$,
\[
   \sum_{l=0}^{n-1} \lambda^{(n)}_{l} = 1 .
\]
Hence $S_n/M_n$ is a finite convex combination of the points $v(t_l)\in\mathcal K$,
and the convexity of $\mathcal K$ gives $S_n/M_n \in \mathcal K$ for all $n\ge n_0$.
Passing to the limit $n\to\infty$, we obtain
\[
   \frac{S_n}{M_n} \; \longrightarrow \; \frac{1}{m}\int_a^b w(t)\,v(t)\,\d t,
\]
and since $\mathcal K$ is closed, the limit belongs to $\mathcal K$.
Multiplying by $m$ concludes the proof.
\end{proof}

Let us now define the set of retarded velocity trajectories $\HH(t)$, its convex hull $\KK(t)$ and its diameter $\LL(t)$,
\begin{equation}\label{def:HKL}
    \HH(t):=\bigcup_{i\ne j}\{v_j(s):r_{ij}(t)\leq s\leq t\},\qquad   \KK(t):=\co\HH(t),\qquad \LL(t):=\diam\HH(t).
\end{equation}
Obviously, for all $t\geq 0$ the set $\HH(t)$ is compact and has the same diameter as its compact convex hull $\KK(t)$.
The following lemma establishes nonexpansivity of $\KK(t)$ along the solutions of \eqref{eq:CS1}--\eqref{cond:sublum}.
The proof relies on the causality established in Lemma \ref{lem:causality}.
However, it does \emph{not} require a priori boundedness of the delays.

\begin{lemma}[Nonexpansivity of the velocity convex hull]
\label{lem:nonexp}
Let $(\solx(t),\solv(t))_{t\geq 0}$ be a solution of \eqref{eq:CS1}--\eqref{cond:sublum}.
Then
\(   \label{eq:nonexp}
   \KK(t)\subseteq\KK(t_0)   \qquad\mbox{for every }0\leq t_0\leq t.
\)
Consequently,
\begin{equation}\label{eq:L-nonincreasing}
   \LL(t)\leq\LL(t_0)   \qquad\mbox{for every } 0\leq t_0\leq t.
\end{equation}
\end{lemma}

\begin{proof}
Fix $t_0\geq0$ and set $\KK_0 := \KK(t_0)$.
We first prove that
\begin{equation}\label{eq:future-velocities-in-K}
   v_i(t)\in \KK_0   \qquad\mbox{for every }t\geq t_0   \mbox{ and }i\in[N].
\end{equation}
For $i\neq j$, introduce
\(  \label{def:psiij}
   \psi_{ij}(t)   :=      \psi\left(   |x_j(r_{ij}(t))-x_i(t)|   \right),
   \qquad
   \psi_i(t):=\sum_{j\neq i}\psi_{ij}(t).
\)
Then \eqref{eq:CS2} can be written as
\[
   \dot v_i(t)   =  -\psi_i(t)v_i(t)  +  \sum_{j\neq i}\psi_{ij}(t)v_j(r_{ij}(t)).
\]
For $t\geq s\geq t_0$, define
\(  \label{eq:Ei}
   E_i(t,s)   :=  \exp\left( -\int_s^t \psi_i(\sigma)\,\d \sigma   \right).
\)
The variation-of-constants formula gives
\begin{equation}\label{eq:voc}
   v_i(t) =   E_i(t,t_0)v_i(t_0)+   \sum_{j\neq i}   \int_{t_0}^t  E_i(t,s)\psi_{ij}(s)  v_j(r_{ij}(s))\,\d s.
\end{equation}
Note that $\frac{\partial}{\partial s}E_i(t,s)   =   \psi_i(s)E_i(t,s)$, hence
\[
   \int_{t_0}^tE_i(t,s)\psi_i(s)\,\d s   =   E_i(t,t)-E_i(t,t_0)   =   1-E_i(t,t_0),
\]
and we have
\begin{equation}\label{eq:nonexp-normalization}
   E_i(t,t_0) +   \sum_{j\neq i}   \int_{t_0}^tE_i(t,s)\psi_{ij}(s)\,\d s =  E_i(t,t_0)   +  \int_{t_0}^tE_i(t,s)\psi_i(s)\,\d s  =1.
\end{equation}

Fix now an arbitrary $T>t_0$ and introduce the compact convex set
\begin{equation}\label{def:nonexp-CT}
   \mathcal{C}_T  :=   \co\left(   \KK_0 \cup   \{v_i(s);\,i\in[N],\ t_0\leq s\leq T\}   \right).
\end{equation}
We claim that every velocity appearing on the right-hand side of
\eqref{eq:voc}, for $t\in[t_0,T]$, belongs to $\mathcal{C}_T$.
Indeed, if $r_{ij}(s)\geq t_0$, then $t_0\leq r_{ij}(s)\leq s\leq T$
and therefore $v_j(r_{ij}(s))\in \mathcal{C}_T$
by the definition of $\mathcal{C}_T$.
On the other hand, if $r_{ij}(s)<t_0$, the monotonicity of $r_{ij}$, implied by causality,
gives $r_{ij}(t_0)\leq r_{ij}(s)<t_0$.
Consequently,
\[
   v_j(r_{ij}(s))\in\HH(t_0)\subseteq \KK_0\subseteq \mathcal{C}_T.
\]
Moreover, by definition, $v_i(t_0)\in\HH(t_0)\subseteq \KK_0$. It now follows from
\eqref{eq:nonexp-normalization},
by applying Lemma~\ref{lem:convex-integral} to each of the integrals in \eqref{eq:voc}, that
\begin{equation}\label{eq:nonexp-first-combination}
   v_i(t)   \in   E_i(t,t_0)\KK_0   +   \bigl(1-E_i(t,t_0)\bigr)\mathcal{C}_T,   \qquad t\in[t_0,T].
\end{equation}
Note that
\[
  1 \geq E_i(t,t_0)   \geq   e^{-(N-1)(T-t_0)}     =:1-q \in(0,1).
\]
Relation \eqref{eq:nonexp-first-combination} implies that for every $t\in[t_0,T]$ there exist $k_t\in \KK_0$ and $c_t\in \mathcal{C}_T$ such that
\[  \begin{aligned}
   v_i(t)  &=  E_i(t,t_0) k_t   +   \bigl(1-E_i(t,t_0)\bigr) c_t \\
   &= (1-q) k_t + q  \left[  \frac{E_i(t,t_0)-1+q}{q} k_t   +  \frac{1-E_i(t,t_0)}{q} c_t   \right].
\end{aligned}  \]
Since $\KK_0\subseteq \mathcal{C}_T$, the expression in the brackets is a convex combination of elements of $\mathcal{C}_T$
and thus itself an element of $\mathcal{C}_T$.
Therefore,
\(    \label{eq:incl1}
   v_i(t)\in (1-q) \KK_0+q\mathcal{C}_T   \qquad\mbox{for all } t\in[t_0,T].
\)
Moreover, $\KK_0\subseteq \mathcal{C}_T$ also implies
\(    \label{eq:incl2}
   \KK_0 \subseteq (1-q) \KK_0 + q\mathcal{C}_T
\)
Using \eqref{eq:incl1} and \eqref{eq:incl2} in \eqref{def:nonexp-CT} gives
\begin{equation}\label{eq:nonexp-self-inclusion}
   \mathcal{C}_T  =   \co\left(   \KK_0\cup   \{v_i(s);\,i\in[N],\ t_0\leq s\leq T\}   \right)    \subseteq (1-q) \KK_0+q\mathcal{C}_T,
\end{equation}
where we also used the convexity of the right-hand set.

Iterating
\eqref{eq:nonexp-self-inclusion} and using the convexity of $\KK_0$ gives
\begin{equation}\label{eq:nonexp-iteration}
   \mathcal{C}_T  \subseteq   (1-q )\KK_0+q \left[(1-q) \KK_0+q\mathcal{C}_T\right]=(1-q^2)\KK_0+q^2\mathcal{C}_T\subseteq \dots \subseteq (1-q^n)\KK_0+q^n\mathcal{C}_T,
\end{equation}
for all $n\in\N$. To pass to the limit $n\to\infty$, fix any $z\in \mathcal{C}_T$.
Then \eqref{eq:nonexp-iteration} implies that for every $n\in\N$, there exist $k_n\in \KK_0$ and
$c_n\in \mathcal{C}_T$ such that
\[
   z=(1-q^n)k_n+q^nc_n.
\]
Since $\KK_0$ is compact, along a subsequence one has $k_n\to k\in \KK_0$.
Moreover, $\mathcal{C}_T$ is bounded and $q^n\to0$, so that $q^nc_n\to0$.
Passing to the limit gives $z=k\in \KK_0$. Hence $\mathcal{C}_T\subseteq \KK_0$,
and the reverse inclusion implied by \eqref{def:nonexp-CT} gives $\mathcal{C}_T=\KK_0$.
Since $T>t_0$ was arbitrary, this proves \eqref{eq:future-velocities-in-K}.

To conclude the proof, it remains to include the complete velocity histories from $\HH(t)$ for $t\geq t_0$.
Therefore, fix $t\geq t_0$ and let $v_j(s)\in\HH(t)$. Then, by the definition \eqref{def:HKL}, for some $i\neq j \in [N]$,
\[
   r_{ij}(t)\leq s\leq t.
\]
If $s\geq t_0$, relation \eqref{eq:future-velocities-in-K} directly gives
$v_j(s)\in \KK_0$. On the other hand, if $s<t_0$, the monotonicity of $r_{ij}$ yields
\[
   r_{ij}(t_0)\leq r_{ij}(t)\leq s<t_0,
\]
and hence, by \eqref{def:HKL}, $v_j(s)\in\HH(t_0)\subseteq \KK_0$.
We have therefore proved
\[
   \HH(t)\subseteq \KK_0=\KK(t_0) \qquad\mbox{for all } t\geq t_0,
\]
and taking the convex hulls gives \eqref{eq:nonexp}.
Finally, \eqref{eq:L-nonincreasing} follows directly from
\[
   \diam\KK(t)=\diam\HH(t)=\LL(t).
\]
\end{proof}

Next we show that the nonexpansivity established in Lemma \ref{lem:nonexp} can be strengthened
to strict contractivity as soon as the delays are uniformly bounded.

\begin{lemma}[Strict contraction under bounded delays]
\label{lem:general-exponential}
Assume the hypotheses of Lemma~\ref{lem:nonexp}, and assume that, for some constant $\overline\tau\geq0$
\begin{equation}\label{eq:general-uniform-bounds}
    \tau_{ij}(t)\leq\overline\tau 
    \qquad
    \mbox{for all } t\geq 0 \mbox{ and } i\ne j.
\end{equation}
Denote
\begin{equation} \label{eq:general-contraction-constants}
    T:=2\overline\tau + 1,
    \qquad
    \eta:= N  \psi\left(  \cvel\overline\tau  \right)  e^{-NT}.
\end{equation}
Then $0<\eta<1$, and for any $t_0 \geq 0$,
\begin{equation}\label{eq:general-block-contraction}
    \LL(t_0+T)
    \leq
    (1-\eta)\LL(t_0). 
\end{equation}
\end{lemma}

\begin{proof}
We use the notation and results of the proof of Lemma~\ref{lem:nonexp}.
In particular, \eqref{eq:general-uniform-bounds} implies that $\psi_{ij}$ defined in \eqref{def:psiij} satisfy $\psi_{ij}(t) \geq \pp>0$ where we used \eqref{def:tauij} and the monotonicity and global positivity of the influence function $\psi$. 

For a fixed $t_0\geq0$ we set $\KK_0:=\KK(t_0)$,
and Lemma~\ref{lem:nonexp} gives $\KK(t)\subseteq \KK_0$ for all $t\geq t_0$. Arguing as in the derivation of \eqref{eq:nonexp-first-combination}, we obtain
\begin{equation}\label{eq:general-retained-atom}
    v_i(t)    \in    E_i(t,t_0) v_i(t_0)    +    \bigl(1-E_i(t,t_0)\bigr) \KK_0    \qquad\mbox{for all } t\geq t_0,
\end{equation}
where $E_i(t,t_0)$ given by \eqref{eq:Ei} satisfies
\begin{equation}\label{eq:general-E-lower}
    E_i(t,t_0) \geq e^{-N(t-t_0)}.
\end{equation}
For $u\geq t_0+\overline\tau$, the variation-of-constants formula \eqref{eq:voc} gives
\begin{equation}\label{eq:general-voc-expanded}
\begin{aligned}
    v_i(u)  &=   E_i(u,t_0) v_i(t_0)    +    \sum_{j\ne i}\int_{t_0}^{t_0+\overline\tau}    E_i(u,s)\psi_{ij}(s)v_j(r_{ij}(s))\,\d s
    \\
    &\qquad\qquad +    \sum_{j\ne i}\int_{t_0+\overline\tau}^{u}    E_i(u,s)\psi_{ij}(s)v_j(r_{ij}(s))\,\d s.
\end{aligned}
\end{equation}
For $s\geq t_0+\overline\tau$, the uniform bound on $\tau_{ij}$ imposed by assumption \eqref{eq:general-uniform-bounds} gives
\[
    r_{ij}(s)\geq s-\overline\tau\geq t_0.
\]
We may therefore apply \eqref{eq:general-retained-atom} at the time
$r_{ij}(s)$. Thus, for every $j\ne i$, there exists
$\zeta_{ij}(s)\in \KK_0$ such that
\begin{equation}\label{eq:general-retained-atom-pointwise}
    v_j(r_{ij}(s)) =   E_j(r_{ij}(s),t_0) v_j(t_0)    +    \bigl(1-E_j(r_{ij}(s),t_0)\bigr)\zeta_{ij}(s).
\end{equation}
Substituting \eqref{eq:general-retained-atom-pointwise} into
\eqref{eq:general-voc-expanded} gives
\begin{align*}
    v_i(u) &=   E_i(u,t_0) v_i(t_0)   +  \sum_{j\ne i}\int_{t_0}^{t_0+\overline\tau}  E_i(u,s)\psi_{ij}(s)v_j(r_{ij}(s))\,\d s    \\
    &\quad+    \sum_{j\ne i}    \left(
        \int_{t_0+\overline\tau}^{u}    E_i(u,s)\psi_{ij}(s)    E_j(r_{ij}(s),t_0)\,\d s    \right) v_j(t_0)    
    \\
    &\quad+    \sum_{j\ne i}\int_{t_0+\overline\tau}^{u}    E_i(u,s)\psi_{ij}(s)
        \bigl(1-E_j(r_{ij}(s),t_0)\bigr)  \zeta_{ij}(s)\,\d s.
\end{align*}
Define $\alpha_{ii}(u):=E_i(u,t_0)$ and, for $j\ne i$,
\(  \label{def:alphaij}
    \alpha_{ij}(u)  :=  \int_{t_0+\overline\tau}^{u}   E_i(u,s)\psi_{ij}(s)  E_j(r_{ij}(s),t_0)\,\d s.
\)
Then
\begin{equation}\label{eq:general-explicit-remainder}
    v_i(u)  =   \sum_{j=1}^N\alpha_{ij}(u) v_j(t_0) + R_i(u),
\end{equation}
where
\begin{align*}
    R_i(u)   &:=    \sum_{j\ne i}\int_{t_0}^{t_0+\overline\tau}   E_i(u,s) \psi_{ij}(s) v_j(r_{ij}(s))\,\d s
    \\
    &\quad+   \sum_{j\ne i}\int_{t_0+\overline\tau}^{u}   E_i(u,s)\psi_{ij}(s)   \bigl(1-E_j(r_{ij}(s),t_0)\bigr)    \zeta_{ij}(s)\,\d s.
\end{align*}
Every velocity $v_j(r_{ij}(s))$ appearing in the first integral is an element of $\KK_0$ by Lemma~\ref{lem:nonexp}, and
$\zeta_{ij}(s)\in \KK_0$ in the second integral.

Next we calculate the total mass of the coefficients of $v_j(r_{ij}(s))$ and $\zeta_{ij}(s)$ in $R_i(u)$.
The normalization \eqref{eq:nonexp-normalization} with $t:=u$ gives
\[
    E_i(u,t_0)  +    \sum_{j\ne i}\int_{t_0}^{u}    E_i(u,s)\psi_{ij}(s)\,\d s   =1.
\]
Consequently, recalling that $\alpha_{ii}(u)=E_i(u,t_0)$ and \eqref{def:alphaij}, we have
\begin{align*}
    &\sum_{j\ne i}\int_{t_0}^{t_0+\overline\tau}    E_i(u,s)\psi_{ij}(s)\,\d s+    \sum_{j\ne i}\int_{t_0+\overline\tau}^{u}   E_i(u,s)\psi_{ij}(s)
         \bigl(1-E_j(r_{ij}(s),t_0)\bigr)\,\d s\\
         &=\sum_{j\ne i}\int_{t_0}^{u}    E_i(u,s)\psi_{ij}(s)\,\d s-  \sum_{j\ne i}\alpha_{ij}(u)\\
         &=1-\sum_{j=1}^N\alpha_{ij}(u)      
\end{align*}
Thus $R_i(u)$ is a nonnegative combination of points of $\KK_0$ with
coefficients summing to $1-\sum_{j=1}^N\alpha_{ij}(u)$. Since $\KK_0$ is convex, there
exists $\zeta_i(u)\in \KK_0$ such that
\[
    R_i(u)   =    \left(1-\sum_{j=1}^N \alpha_{ij}(u)\right)\zeta_i(u).
\]
Equation \eqref{eq:general-explicit-remainder} therefore yields
\(  \label{eq:viu}
    v_i(u)  =    \sum_{j=1}^N\alpha_{ij}(u) v_j(t_0)
        + \left(1-\sum_{j=1}^N\alpha_{ij}(u)\right)\zeta_i(u).
\)
Consider now $v_i(u)\in \HH(t_0+T)$ where $u\in[t_0+T-\overline\tau,t_0+T]$. Then \eqref{eq:general-E-lower} gives
\[
    \alpha_{ii}(u) = E_i(u,t_0) \geq e^{-NT}.
\]
For $j\ne i$, using $r_{ij}(s)\leq s$ and \eqref{eq:general-E-lower}, \eqref{def:alphaij}, we obtain
\begin{align*}
    \alpha_{ij}(u)
    &\geq
    \pp\int_{t_0+\overline\tau}^u
    e^{-N(u-s)}
    e^{-N(r_{ij}(s)-t_0)}\,\d s
    \\
    &\geq
    \pp\int_{t_0+\overline\tau}^u e^{-N(u-t_0)}\,\d s
    \\
    &=
    \pp(u-t_0-\overline\tau)e^{-N(u-t_0)}
    \\
    &\geq {\pp} e^{-NT},
\end{align*}
where we used the definition $T=2\overline\tau+1$ in the last estimate.
Since $\pp \leq 1$, we can write $\alpha_{ij}(u) \geq {\pp} e^{-NT}$
for all $i$, $j\in [N]$, including $i=j$.

Define $\beta_{ij}(u):=\alpha_{ij}(u)-{\pp} e^{-NT} \geq 0$. Then \eqref{eq:viu} can be written as
\begin{align*}
    v_i(u) &=  {\pp} e^{-NT}\sum_{j=1}^N v_j(t_0)   +    \sum_{j=1}^N\beta_{ij}(u)v_j(t_0)
    +   \left(1-\sum_{j=1}^N\alpha_{ij}(u)\right)\zeta_i(u).
\end{align*}
The coefficients in the last two terms are nonnegative, and their sum is
\[
    \sum_{j=1}^N\beta_{ij}(u)    +    1-\sum_{j=1}^N\alpha_{ij}(u)
    =1-N{\pp} e^{-NT}
    = 1-\eta
    > 0
\]
since $\eta\leq Ne^{-N}\leq  e^{-1}<1$.
 Therefore, \eqref{eq:viu} becomes
\[
    v_i(u)   =    {\pp} e^{-NT}\sum_{j=1}^N v_j(t_0)   + \left( 1-\eta \right) \widehat\zeta_i(u),
\]
with 
\[
    \widehat\zeta_i(u)   := \left( 1-\eta \right)^{-1}
       \left[   \sum_{j=1}^N\beta_{ij}(u)v_j(t_0)    +    \left(1-\sum_{j=1}^N\alpha_{ij}(u)\right)\zeta_i(u)  \right].
\]
Since $v_i(t_0)\in \KK_0$, $\zeta_i(u)\in \KK_0$ for all $i\in[N]$, and $\KK_0$ is convex, the vectors
$\widehat\zeta_i(u)$ belong to $\KK_0$ for all $i\in[N]$.
Consequently, we have
\begin{equation*}
    v_i(u)   \in  \frac{\eta}{N}\sum_{j=1}^N v_j(t_0) + (1-\eta) \KK_0,
\end{equation*}
taking convex hull of $\HH(t_0+T)$, it gives
\begin{equation*} 
    \KK(t_0+T)    \subseteq    \frac{\eta}{N}\sum_{i=1}^N v_i(t_0) + (1-\eta)\KK(t_0),
\end{equation*}
and taking diameters finally proves \eqref{eq:general-block-contraction} since the diameter is invariant under translations.
\end{proof}

\begin{remark}
The contraction estimate \eqref{eq:general-block-contraction} could be used
for the usual bootstrap argument, like the one used in \cite{H:KRM, Cartabia},
to prove flocking in the sense of Definition \ref{def:flocking}.
In particular, one would use the bound $\dx(t)\leq\dx(0)+\int_0^t\dv(s)\,\d s$,
combined with $\dv(t)\leq\LL(t)$,
and the bound $\tau_{ij}(t) \leq \frac{\dx(t)}{\cvel-\svel}$ provided by \eqref{eq:delay-distance}.
However, this leads to a very restrictive small-data/large propagation speed criterion for flocking,
in the spirit of the one derived in \cite[Lemma 7]{H:KRM}.
As the goal of this paper is to provide an optimal flocking condition,
we do not follow this strategy here.
\end{remark}

Instead, in the rest of the paper we attempt to prove that the delays (and thus the position diameter) are indeed uniformly bounded,
so that Lemma \ref{lem:general-exponential} directly implies exponential convergence of the velocity diameter to zero.
Indeed, iteration at the times $nT$, with $T$ given by \eqref{eq:general-contraction-constants},
together with the nonexpansivity of $\LL$ established in Lemma \ref{lem:nonexp}, yields
\begin{equation*} 
   \dv(t) \leq \LL(t)    \leq    (1-\eta)^{\lfloor t/T\rfloor}\LL(0)  \qquad \mbox{for all } t\geq0.
\end{equation*}

\section{Two agents in one spatial dimension}\label{sec:1D}

In this section we establish the uniform boundedness of the delays for solutions of \eqref{eq:CS1}--\eqref{IC},
asserting only the minimal assumptions \eqref{cond:sublum}, \eqref{IC:comp} and the fat-tail condition \eqref{cond:flocking}.
Using the notation \eqref{def:psiij}, system \eqref{eq:CS1}--\eqref{eq:CS2} for $i\in\{1,2\}$ is written as
\(  \label{eq:CS1D}
   \dot x_i = v_i, \qquad
   \dot v_i = \psi_{ij} \left( v_j(r_{ij}) - v_i \right) \qquad\mbox{with } i \neq j\in \{1, 2\}.
\)
We introduce the function
\begin{equation}\label{def:Psic}
 \Psi_{\cvel}(\rho):=\int_0^\rho\psi(\cvel s)\,\d s
 =\frac1{\cvel}\int_0^{\cvel\rho}\psi(r)\,\d r,\qquad \rho\geq0,
\end{equation}


\begin{lemma}\label{lem:doppler}
Fix $i\ne j$ and let $I$ be a connected component of the set
\[
   \{t\geq 0; \; x_i(t)-x_j(r_{ij}(t)) \neq 0 \} = \{t\geq0; \; \tau_{ij}(t)>0\}.
\]
Then the sign
\[
   \varepsilon_{ij}:=\sgn\bigl(x_i(t)-x_j(r_{ij}(t))\bigr)\in\{-1,1\}
\]
is constant on $I$. For $t\in I$, define
\begin{equation}\label{def:J1D}
    \JJ_{ij}(t):=\Psi_{\cvel}(\tau_{ij}(t))  -\ln\left(\frac{\cvel-\varepsilon_{ij}v_i(t)}{\cvel}\right).
\end{equation}
Then for all $t\in I$ we have
\begin{equation}\label{eq:doppler-dissipation}
    \frac{\d}{\d t}\JJ_{ij}(t)  =-\psi_{ij}(t)\frac{|v_i(t)-v_j(r_{ij}(t))|^2}  {(\cvel-\varepsilon_{ij}v_i(t))(\cvel-\varepsilon_{ij}v_j(r_{ij}(t)))} \leq0
\end{equation}
\end{lemma}

\begin{proof}
Since $| x_i(t)-x_j(r_{ij}(t))| = \cvel\tau_{ij}(t) >0$ for all $t\in I$, the sign $\varepsilon_{ij}$ is obviously constant on $I$.
Also note that $\tau_{ij}$ is continuously differentiable on $I$.

Denote
\[
 \ p(t):=\cvel-\varepsilon_{ij}v_j(r_{ij}(t)),\qquad   q(t):=\cvel-\varepsilon_{ij}v_i(t).
\]
Then \eqref{eq:CS1D} gives $\dot q = \psi_{ij} (p-q)$.
Differentiating \eqref{def:tauij} on $I$, using the compatibility condition \eqref{IC:comp} if $r_{ij}\leq 0$, gives
\[
 \cvel\dot\tau_{ij} = \varepsilon_{ij}\bigl(v_i-v_j(r_{ij})(1-\dot\tau_{ij})\bigr) = p - q + (\cvel -p) \dot\tau_{ij},
\]
and hence $p\dot\tau_{ij} =p-q$.
The chain rule applied to \eqref{def:Psic} yields
\[
    \frac{\d}{\d t}\Psi_{\cvel}(\tau_{ij})= \psi(\cvel \tau_{ij}) \dot\tau_{ij} = \psi_{ij}\frac{p-q}{p}.
\]
Consequently,
\[
   \frac{\d}{\d t}\JJ_{ij} = \psi_{ij}\frac{p-q}{p} - \frac{\dot q}{q} = - \psi_{ij} \frac{(p-q)^2}{pq},
\]
which is \eqref{eq:doppler-dissipation}.
\end{proof}

\begin{corollary}\label{corr:delay-bound}
Along the solutions of \eqref{eq:CS1D}, we have
\begin{equation}\label{eq:potential-bound}
 \Psi_{\cvel}(\tau_{ij}(t))\leq\Psi_{\cvel}(\tau_{ij}(0)) + \ln \frac{\cvel+\svel}{\cvel-\svel} \qquad\mbox{for all } t\geq0,
\end{equation}
In particular, if \eqref{cond:flocking} holds, we have the uniform bound
\begin{equation}\label{eq:delay-bound-constant}
    \tau_{ij}(t) \leq \Psi_{\cvel}^{-1} \left( \Psi_{\cvel}(\tau_{ij}(0)) + \ln\frac{\cvel+\svel}{\cvel-\svel} \right) \qquad\mbox{for all } t\geq 0.
\end{equation}
\end{corollary}

\begin{proof}
If $t\geq 0$ is such that $\tau_{ij}(t)=0$, then the left-hand side in \eqref{eq:potential-bound} vanishes and the inequality holds trivially.
It also holds trivially for $t=0$.

Let us therefore fix $t > 0$ such that $\tau_{ij}(t)>0$ and let $I$ be the maximal connected component of the set $\{s\geq0:\tau_{ij}(s)>0\}$ such that $t\in I$.
Let $s_I\geq0$ be the left endpoint of the closure of $I$. Integrating \eqref{eq:doppler-dissipation} over the interval $(s_I, t)$ gives
\begin{align*}
 \Psi_{\cvel}(\tau_{ij}(t))
 &\leq\Psi_{\cvel}(\tau_{ij}(s_I))
 +\ln\frac{\cvel-\varepsilon_{ij}v_i(t)}{\cvel-\varepsilon_{ij}v_i(s_I)}\nonumber\\
 &\leq\Psi_{\cvel}(\tau_{ij}(0))+ \ln \frac{\cvel+\svel}{\cvel-\svel},
\end{align*}
where we used that if $s_I>0$, then by definition $\tau_{ij}(s_I)=0$, and $\Psi_{\cvel}(0) = 0$.
Consequently, we have $\Psi_{\cvel}(\tau_{ij}(s_I)) \leq \Psi_{\cvel}(\tau_{ij}(0))$ and \eqref{eq:potential-bound} follows.

Due to the global positivity of $\psi$ and \eqref{cond:flocking}, the function $\Psi_{\cvel}$ is strictly increasing on $[0,\infty)$.
Consequently, its inverse is well defined and \eqref{eq:delay-bound-constant} follows directly from \eqref{eq:potential-bound}.
\end{proof}

\begin{remark}
The form of the functional \eqref{def:J1D} is motivated by the classical Doppler effect.
On an interval where the signal orientation $\varepsilon_{ij}$ is constant, differentiating the light-cone relation \eqref{def:tauij} gives
\begin{equation*}
   \dot r_{ij}(t) =\frac{\mathfrak c-\varepsilon_{ij}v_i(t)} {\mathfrak c-\varepsilon_{ij}v_j(r_{ij}(t))}
   =\left(1-\frac{\varepsilon_{ij}v_i(t)}{\cvel}\right) \left(1-\frac{\varepsilon_{ij}v_j(r_{ij}(t))}{\cvel}\right)^{-1}.
\end{equation*}
If agent $j$ emits a signal with phase $\theta_j(t)$, its phase received by $i$ is $\theta_j(r_{ij}(t))$.
Therefore, the received and emitted instantaneous frequencies satisfy $\nu_i(t)=\dot r_{ij}(t)\nu_j(r_{ij}(t))$.
Thus $\dot r_{ij}(t) $ is the Doppler frequency ratio in the common time coordinate of the model.
Its first factor accounts for the motion of the receiver through the incoming wavefronts:
if their local spacing is $\mu$, their crossing rate at the receiver is $(\mathfrak c-\varepsilon_{ij}v_i)/\mu$,
compared with $\mathfrak c/\mu$ at a stationary receiver. Their ratio is therefore $1-\varepsilon_{ij}v_i/\mathfrak c$.
Similarly, the second factor accounts for the modification of wavefront spacing caused by the source motion;
 see \cite[Sec.~17.7, Eqs.~(17.19)--(17.20)]{MoebsLingSanny2016}.
 Therefore, the logarithmic term in \eqref{def:J1D}
is the negative logarithm of the receiver factor. The derivative can also be written as
\begin{equation*}
    \frac{\d}{\d t}\JJ_{ij}(t) =-\psi_{ij}(t)\frac{(\dot r_{ij}(t)-1)^2}{\dot r_{ij}(t)}. 
\end{equation*}
\end{remark}

\section{Two agents in multiple spatial dimensions}
\label{sec:multid-flocking}

In this section we extend the result of Section~\ref{sec:1D} to
multiple spatial dimensions $d\geq 2$. We retain the subluminality,
compatibility and fat-tail assumptions \eqref{cond:sublum}, \eqref{IC:comp} and \eqref{cond:flocking},
and, in addition, assume that all values of the initial velocity histories lie
on one common affine line. Choose a unit vector $e\in\R^d$ and
write this line as $w+\R e$, where $w\cdot e=0$.
Thus, with $X_i:=x_i^0(0)$, the initial histories have the form
\begin{equation}\label{eq:multid-history}
\begin{aligned}
    v_i^0(t)&=w+u_i^0(t)e,\\
    x_i^0(t)&=X_i+tw+e\int_0^t u_i^0(s)\,\d s,
\end{aligned}
\qquad t\in[-S_0,0],\quad i=1,2,
\end{equation}
where $u_i^0$ are continuous scalar functions and $S_0$ is the
relevant history length specified in Proposition~\ref{thm:ex}.
The position histories in \eqref{eq:multid-history} follow from
\eqref{IC:comp}. 

Define
\begin{equation*}
    P_\perp:=I-e\otimes e,
    \qquad
    u_i(t):=e\cdot v_i(t),
\end{equation*}
and put
\begin{equation}\label{eq:CU}
    C:=\sqrt{\cvel^2-|w|^2},
    \qquad
    U:=\max_{i=1,2}\max_{s\in[-S_0,0]}|u_i^0(s)|.
\end{equation}
We first show that the common affine velocity line is preserved.

\begin{lemma}\label{lem:multid-invariance}
For every $t\in[-S_0,\infty)$ and $i=1,2$, we have
\begin{equation}\label{eq:multid-affine-line}
    v_i(t)=w+u_i(t)e,
    \qquad
    v_1(t)-v_2(t)=(u_1(t)-u_2(t))e.
\end{equation}
Moreover, $|u_i(t)|\leq U<C$ on this time interval.
In particular,
\begin{equation*}
    \dv(t)=|u_1(t)-u_2(t)|,
    \qquad t\geq0.
\end{equation*}
\end{lemma}

\begin{proof}
The assertions about the affine line and the bound $|u_i|\leq U$
hold on the initial interval by \eqref{eq:multid-history}.
Since the segment
\begin{equation*}
    \mathcal S:=\{w+ue;\; |u|\leq U\}
\end{equation*}
is convex and contains all relevant initial velocity values, the definition
\eqref{def:HKL} gives $\KK(0)\subseteq\mathcal S$.
For $t\geq0$, Lemma~\ref{lem:nonexp} therefore implies
\begin{equation*}
    v_i(t)\in\HH(t)\subseteq\KK(t)
    \subseteq\KK(0)\subseteq\mathcal S.
\end{equation*}
Consequently, $P_\perp v_i(t)=w$ and $|u_i(t)|\leq U$,
which proves \eqref{eq:multid-affine-line} and the formula for $\dv$.

Finally, orthogonality of $w$ and $e$ gives
\[
    |w|^2+U^2
    =\max_{i=1,2}\max_{s\in[-S_0,0]}|v_i^0(s)|^2
    \leq\svel^2.
\]
Thus $C^2-U^2  =\cvel^2-|w|^2-U^2  \geq\cvel^2-\svel^2>0$,
and hence $U<C$.
\end{proof}

The compatibility condition and Lemma~\ref{lem:multid-invariance}
give, for every $t\in[-S_0,\infty)$,
\begin{equation*}
    P_\perp x_i(t)
    =
    P_\perp X_i+\int_0^t P_\perp v_i(s)\,\d s
    =
    P_\perp X_i+tw.
\end{equation*}
Consequently, defining $y_i(t):=e\cdot x_i(t)$, we obtain
\begin{equation*}
    x_i(t)= (e\cdot x_i(t))e+P_\perp x_i(t) = y_i(t)e+P_\perp X_i+tw,
    \qquad t\in[-S_0,\infty).
\end{equation*}
Projecting \eqref{eq:CS1}--\eqref{eq:CS2} onto $e$ yields, for $t\geq0$,
\begin{equation}\label{eq:multid-scalar-system}
    \dot y_i(t)=u_i(t),
    \qquad
    \dot u_i(t)
    =
    \psi_{ij}(t)\bigl(u_j(r_{ij}(t))-u_i(t)\bigr),
\end{equation}
where $j\neq i$ and $\psi_{ij}(t)$ given by \eqref{def:psiij}.
The light-cone relation \eqref{def:tauij} becomes
\begin{equation}\label{eq:multid-projected-cone}
    \cvel^2\tau_{ij}(t)^2
    =
    \bigl(y_j(r_{ij}(t))-y_i(t)\bigr)^2
    +
    |z_{ij}+w\tau_{ij}(t)|^2,
\end{equation}
where $z_{ij}:=P_\perp(X_i-X_j)$.

We introduce $\rho_{ij}\geq 0$ as the nonnegative root of the equation
\begin{equation}\label{eq:root}
  \mathfrak c^2\rho_{ij}^2 = |z_{ij}+w\rho_{ij}|^2.
\end{equation}
This leads to
\begin{equation} \label{def:rhoKij}
    \rho_{ij}:=\frac{z_{ij}\cdot w + K_{ij}}{C^2} \qquad\mbox{with}\quad
    K_{ij}:=\sqrt{(z_{ij}\cdot w)^2+C^2|z_{ij}|^2}. 
\end{equation}
By equation \eqref{eq:multid-projected-cone} we readily have
\begin{equation}\label{eq:multid-longitudinal-separation}
    \tau_{ij}(t)\geq\rho_{ij}\qquad \mbox{and}\qquad |y_j(r_{ij}(t))- y_i(t)|=R_{ij}(\tau_{ij}(t)),
\end{equation}
with
\begin{equation*}   
   R_{ij}(\sigma):=\sqrt{C^2 \sigma^2-2 (z_{ij}\cdot w) \sigma  - |z_{ij}|^2}\qquad \mbox{for } \sigma \geq\rho_{ij}.
\end{equation*}
Also note that by \eqref{eq:root}, we have $R_{ij}(\rho_{ij}) = 0$.

Define, for $\sigma>\rho_{ij}$,
\begin{equation*}   
 k_{ij}(\sigma):=R_{ij}'(\sigma)=\frac{C^2\sigma - z_{ij}\cdot w}{R_{ij}(\sigma)}.
\end{equation*}
Then, again by direct calculation,
\begin{equation}\label{eq:multid-k-identities}
  k_{ij}^2=C^2+\frac{K_{ij}^2}{R_{ij}^2},\qquad  
    \int_{\rho_{ij}}^\infty(k_{ij}(s)-C)\,\mathrm ds = \int_{\rho_{ij}}^\infty(R_{ij}(s)-Cs)'\,\mathrm ds =\frac{K_{ij}}{C},
\end{equation}
where for the evaluation of the integral we used 
$R_{ij}(\rho_{ij})=0$ and $R_{ij}(\sigma)-C\sigma\to- (z_{ij}\cdot w)/C$ as $\sigma\to\infty$.

\subsection*{A modified Lyapunov functional}

We now introduce a modified version of the functional $\JJ_{ij}$ defined in \eqref{def:J1D}. Our strategy is to reduce the dynamics to the one-dimensional setting of Section~\ref{sec:1D} and to adapt the functional \eqref{def:J1D} accordingly. In particular, the propagation speed $\mathfrak c$ in the proof of Lemma~\ref{lem:doppler} is replaced by $k_{ij}(\tau_{ij})$.
In the first step, we compensate for this by replacing the constant $\mathfrak c$ in the logarithmic term of \eqref{def:J1D} by the function $Q_{ij}$ of the delay,
constructed in the following lemma.

\begin{lemma}\label{lem:multid-correction}
Assume that \eqref{cond:flocking} holds.
For $\sigma\geq\rho_{ij}$ define
\begin{equation}\label{eq:multid-correction}
    Q_{ij}(\sigma) :=  \int_\sigma^\infty \psi(\cvel s)k_{ij}(s)  \exp\bigl(\Psi_{\cvel}(\sigma)-\Psi_{\cvel}(s)\bigr)\,\mathrm ds, 
\end{equation}
with $\Psi_{\cvel}$ given by \eqref{def:Psic}.
Then $Q_{ij}$ is finite and continuous on $[\rho_{ij},\infty)$, locally Lipschitz on $(\rho_{ij},\infty)$, and satisfies
\begin{equation}\label{eq:multid-correction-bounds}
 C\leq Q_{ij}(\sigma)\leq C+\frac{K_{ij}}{C} \qquad\mbox{for all } \sigma\geq \rho_{ij},
\end{equation}
with the constants $C$ given by \eqref{eq:CU} and $K_{ij}$, $\rho_{ij}$ given by \eqref{def:rhoKij}.
Moreover,
\begin{equation}\label{eq:multid-correction-ode}
    Q_{ij}'(\sigma)=\psi(\cvel \sigma)\bigl(Q_{ij}(\sigma)-k_{ij}(\sigma)\bigr) \qquad\mbox{for all } \sigma > \rho_{ij}.
\end{equation}
\end{lemma}

\begin{proof}
Assumption \eqref{cond:flocking} gives $\Psi_{\cvel}(\infty)=\infty$ and, therefore,
\[
   \int_\sigma^\infty \psi(\cvel s)\exp\bigl( -\Psi_{\cvel}(s) \bigr)\,\mathrm ds =
   \int_{\Psi_\cvel(\sigma)}^\infty e^{-s} \, \d s = e^{-\Psi_\cvel(\sigma)},
\]
where we used $\Psi'_\cvel(\sigma) = \psi(\cvel \sigma)$. Consequently, for all $\sigma\geq\rho_{ij}$ we have
\begin{equation} \label{eq:multid-tail-normalization}
    \int_\sigma^\infty \psi(\cvel s)\exp\bigl(\Psi_{\cvel}(\sigma)-\Psi_{\cvel}(s) \bigr)\,\mathrm ds=1.
\end{equation}

The first identity in \eqref{eq:multid-k-identities} gives $k_{ij} \geq C$, so that with \eqref{eq:multid-tail-normalization} we readily have $Q_{ij}(\sigma) \geq C$.
Moreover, using the global bound $\psi\leq 1$, the monotonicity of $\Psi_{\cvel}$ and \eqref{eq:multid-k-identities},
\begin{equation*}
    Q_{ij}(\sigma)-C \leq \int_\sigma^\infty(k_{ij}(s)-C)\,\mathrm ds  \leq\frac{K_{ij}}{C} \qquad\mbox{for all } \sigma\geq \rho_{ij}.
\end{equation*}
This proves \eqref{eq:multid-correction-bounds}.
Formula \eqref{eq:multid-correction-ode} follows by differentiation of \eqref{eq:multid-correction}.
Finally, on compact intervals contained in $(\rho_{ij},\infty)$ the derivative $Q_{ij}'$ is bounded, proving local Lipschitz continuity of $Q_{ij}$.
\end{proof}

We now establish an analogue of Lemma \ref{lem:doppler} with a modified version of the functional $\JJ_{ij}$ defined in \eqref{def:J1D}.
We note that the modification given by \eqref{eq:multid-functional} below coincides with \eqref{def:J1D} in the one-dimensional setting,
where we have $P_\perp=0$, so that $w=z_{ij}=0$, and, in turn, $Q_{ij} \equiv C \equiv \cvel$.
Consequently, \eqref{def:J1D} is a special case of \eqref{eq:multid-functional} and
this justifies using the same notation $\JJ_{ij}$ for both.

\begin{lemma}\label{lem:multid-doppler}
Fix a connected component $I$ of the set $\{t\geq0:y_i(t)-y_j(r_{ij}(t))\ne0\}$. 
Then the sign $\varepsilon_{ij}:=\operatorname{sgn}(y_i(t)-y_j(r_{ij}(t)))$ is constant on $I$.
For $t\in I$ define
\begin{equation}\label{eq:multid-functional}
   \JJ_{ij}(t):=\Psi_{\cvel}(\tau_{ij}(t))  -\ln\left(\frac{Q_{ij}(\tau_{ij}(t))-\varepsilon_{ij}u_i(t)}{C}\right),
\end{equation}
with $Q_{ij}$ defined in \eqref{eq:multid-correction} and $C$ given by \eqref{eq:CU}.
Then for all $t\in I$ we have
\begin{equation}\label{eq:multid-dissipation}
    \frac{\d}{\d t}\JJ_{ij}(t)  = -\psi_{ij}(t)\frac{|u_i(t)-u_j(r_{ij}(t))|^2}  {(k_{ij}(\tau_{ij}(t))-\varepsilon_{ij}u_j(r_{ij}(t)))  (Q_{ij}(\tau_{ij}(t))-\varepsilon_{ij}u_i(t))}\leq 0.
\end{equation}
\end{lemma}

\begin{proof}
The function $\varepsilon_{ij}$ is continuous and nonzero on $I$. Therefore, it is constant on $I$.
Moreover, from \eqref{eq:multid-projected-cone}--\eqref{eq:multid-longitudinal-separation} we have
\[
   \{t\geq0:y_i(t)-y_j(r_{ij}(t))\ne0\} = \{t\geq0:\tau_{ij}(t)>\rho_{ij}\},
\]
and so $\tau_{ij} > \rho_{ij}$ on $I$.
Set
\begin{equation*}
   p:=k_{ij}(\tau_{ij})-\varepsilon_{ij}u_j(r_{ij}),\qquad q:=Q_{ij}(\tau_{ij})-\varepsilon_{ij}u_i.
\end{equation*}
Note that $p$, $q \geq C-U>0$. Differentiating the signed version of \eqref{eq:multid-longitudinal-separation} gives
\begin{equation*}
 k_{ij}(\tau_{ij})\dot\tau_{ij}
 =\varepsilon_{ij}\bigl(u_i-u_j(r_{ij})(1-\dot\tau_{ij})\bigr),
\end{equation*}
and therefore
\begin{equation*} 
 p\dot\tau_{ij}=\varepsilon_{ij}\left( u_i-u_j(r_{ij}) \right).
\end{equation*}
Using \eqref{eq:multid-correction-ode} and \eqref{eq:multid-scalar-system}, we compute
\begin{align*}
    \dot q   &= \psi_{ij}(Q_{ij}-k_{ij})\dot\tau_{ij}+\varepsilon_{ij} \psi_{ij}\left( u_i-u_j(r_{ij}) \right)\\
    &=\frac{\psi_{ij}\varepsilon_{ij}\left( u_i-u_j(r_{ij}) \right)}{p}  \bigl(Q_{ij}-k_{ij}+p\bigr)\\ 
    &=\frac{\psi_{ij}\varepsilon_{ij}\left( u_i-u_j(r_{ij}) \right)}{p}
    \bigl(Q_{ij}-\varepsilon_{ij}u_j(r_{ij})\bigr).
\end{align*}
Consequently,
\begin{align*}
   \frac{\d}{\d t}\JJ_{ij}(t) &= \psi_{ij}\dot\tau_{ij}-\frac{\dot q}{q}\\
      &=\frac{\psi_{ij}\varepsilon_{ij}\left( u_i-u_j(r_{ij}) \right)}{p}   \left(1-\frac{Q_{ij}-\varepsilon_{ij}u_j(r_{ij})}{Q_{ij}-\varepsilon_{ij}u_i}\right) \\
       &=-\frac{\psi_{ij}|\left( u_i-u_j(r_{ij}) \right)|^2}{pq}.
\end{align*}
\end{proof}

\begin{corollary}
Along the solutions of \eqref{eq:multid-scalar-system}--\eqref{eq:multid-projected-cone}, we have
\(   \label{eq:potential-bound-multiD}
   \Psi_{\cvel}(\tau_{ij}(t)) \leq \Psi_{\cvel}(\tau_{ij}(0)) + \ln\frac{Q_{ij}^*+U}{C-U},
\)
where $Q_{ij}^*:=C+\frac{K_{ij}}{C}$, with the constants $C$, $U$ given by \eqref{eq:CU} and $K_{ij}$ given by \eqref{def:rhoKij}.

In particular, if \eqref{cond:flocking} holds, we have the uniform bound
\begin{equation}\label{eq:delay-bound-constant-multiD}
   \tau_{ij}(t) \leq \Psi_{\cvel}^{-1} \left( \Psi_{\cvel}(\tau_{ij}(0)) + \ln\frac{Q_{ij}^*+U}{C-U} \right) \qquad\mbox{for all } t\geq 0.
\end{equation}
\end{corollary}

\begin{proof}
First note that $0 < C-U\leq Q_{ij}(\tau_{ij})-\varepsilon_{ij}u_i\leq Q_{ij}^*+U$, therefore $$\ln\frac{Q_{ij}^*+U}{C-U} \geq 0.$$
Obviously, \eqref{eq:potential-bound-multiD} holds trivially for $t=0$.
If $t > 0$ is such that $\tau_{ij}(t)=\rho_{ij}$, then it holds due to the monotonicity of $\Psi_\cvel$
and the bound $\rho_{ij} \leq \tau_{ij}(0)$ provided by \eqref{eq:multid-longitudinal-separation}.

Let us now fix $t > 0$ such that $\rho_{ij} < \tau_{ij}(t)$ and let $I$ be the maximal connected component of the set $\{s\geq0; \; \tau_{ij}(s)> \rho_{ij} \}$ such that $t\in I$.
Let $s_I\geq0$ be its left endpoint. Integrating \eqref{eq:multid-dissipation} over the interval $(s_I, t)$ gives
\begin{align*}
 \Psi_{\cvel}(\tau_{ij}(t))
 &\leq\Psi_{\cvel}(\tau_{ij}(s_I))
 +\ln\frac{Q_{ij}(\tau_{ij}(t))-\varepsilon_{ij}u_i(t)}
 {Q_{ij}(\tau_{ij}(s_I))-\varepsilon_{ij}u_i(s_I)}\nonumber\\
 &\leq\Psi_{\cvel}(\tau_{ij}(0))+ \ln\frac{Q_{ij}^*+U}{C-U},
\end{align*}
where we used $\tau_{ij}(s_I)=\rho_{ij}\leq\tau_{ij}(0)$ if $s_I>0$. 

Due to the global positivity of $\psi$ and \eqref{cond:flocking}, the function $\Psi_{\cvel}$ is strictly increasing on $[0,\infty)$.
Consequently, its inverse is well defined and \eqref{eq:delay-bound-constant-multiD} follows directly from \eqref{eq:potential-bound-multiD}.
\end{proof}

\section{Multiple agents in one spatial dimension}\label{sec:N1d}
We now discuss an extension of the results obtained in Section \ref{sec:1D} to the case of multiple agents, $N\geq 2$ and $d=1$.

For each agent, consider the maximal open intervals on which it is leftmost.
These intervals form an at most countable family $(I_n)_{n\in\N}$, with corresponding indices $\ell_n\in[N]$ such that
\[
    x_{\ell_n}(t)=\min_{i\in[N]}x_i(t),
    \qquad t\in I_n.
\]
We call changes between distinct leftmost trajectories \emph{switches}.
Whenever the leftmost agent switches from $\ell$ to $k$ with $\ell\neq k$ at a time $t_*$, continuity gives
\[
    x_\ell(t_*)=x_k(t_*)    =   \min_{i\in[N]}x_i(t_*),
\]
and, for all $\varepsilon>0$, there is $t\in(t_*-\varepsilon,t_*+\varepsilon)$ such that $x_\ell(t)\neq x_k(t)$.

\begin{lemma}\label{lem:acc}
The switching times can have at most $N-1$ finite accumulation points in $(0,\infty)$.
\end{lemma}

\begin{proof}
Let $(t_n)_{n\in\N}$ be a sequence of distinct genuine switching times
converging to $t_*\in(0,\infty)$. After passing to a subsequence,
the sequence approaches $t_*$ from one side.
Since there are only finitely many pairs of agents, after passing to another subsequence,
every switch occurs between the same two agents $\ell\ne k$. Consequently,
\[
    x_\ell(t_n)=x_k(t_n)=\min_{i\in[N]}x_i(t_n).
\]
Continuity gives $x_\ell(t_*)=x_k(t_*)$, and then the light-cone relation \eqref{def:tauij} implies
\[
    \tau_{\ell k}(t_*) = \tau_{k\ell}(t_*) =0,
    \qquad
    \tau_{\ell j}(t_*) = \tau_{kj}(t_*)    \quad\text{for every }j\notin\{\ell,k\}.
\]
Moreover, writing $\xi:=x_\ell-x_k$, differentiability yields
\[
    v_\ell(t_*)-v_k(t_*) = \dot\xi(t_*)  = \lim_{n\to\infty}  \frac{\xi(t_n)-\xi(t_*)}{t_n-t_*}  = 0.
\]
Consequently, the equations for $v_\ell$ and $v_k$ have identical right-hand
sides at $t_*$. Therefore, the identical continuation
\[
    x_\ell(t)=x_k(t),    \qquad    v_\ell(t)=v_k(t)    \qquad\mbox{for all } t\geq t_*
\]
is a solution of \eqref{eq:CS1}--\eqref{def:tauij},
since for all $t\geq t_*$ the mutual delays and interaction terms between $\ell$ and $k$ vanish,
and delays from the two receivers $\ell$, $k$ to all other emitters agree.
According to Proposition~\ref{thm:ex} this solution is unique.
Thus every finite accumulation point approached from the left coalesces at least two
previously distinct trajectories.
Accumulation at $t_*$ approached from the right is impossible,
because the involved agents coincide for all $t\geq t_*$.
We conclude that with $N$ agents, accumulation of genuine switching times can happen at most $N-1$ times.
\end{proof}

Due to Lemma \ref{lem:acc}, there exists $t_0\geq 0$ such that in $(t_0,\infty)$
either no more switches take place, or there are infinitely many switches with times
accumulating only at infinity.
We start treating the first case, i.e., assume that there exists a fixed agent with index $\ell\in[N]$
which is leftmost forever after $t_0$,
\begin{equation}\label{eq:eventually-leftmost}
    x_\ell(t)=\min_{i\in[N]}x_i(t)    \qquad\text{for all }t\geq t_0.
\end{equation}
For $i\in[N]$, we define
\(   \label{def:FB}
    F_i(t;t_0) := \sum_{j\neq i} \Psi_{\cvel}(\tau_{ij}(t))  - \ln\left(\frac{\cvel+v_i(t)}{\cvel + v_i(t_0)}\right).
\)
Then $F_i$ can be seen as a generalization of the functional $J_{ij}$, given by \eqref{def:J1D}, for the setting with multiple agents.

\begin{lemma} 
\label{lem:eventual-leftmost-delay-bound}
Assume that there exists $t_0\geq 0$ and $\ell\in[N]$ such that \eqref{eq:eventually-leftmost} holds.
Moreover, assume the fat-tail condition \eqref{cond:flocking}.
Then 
\begin{equation*}  
    \max_{i\in[N]} \max_{j\in[N]}\tau_{ij}(t)    \leq     \frac{\cvel+\svel}{\cvel-\svel} \overline\tau \qquad\mbox{for all } t\geq t_0,
\end{equation*}
where
\(  \label{def:otau1}
   \overline\tau := \Psi_{\cvel}^{-1} \left( \sum_{j\neq \ell}  \Psi_{\cvel}(\tau_{\ell j}(t_0)) + \ln\left(\frac{\cvel+\svel}{\cvel -\svel}\right)  \right).
\)

\end{lemma}

\begin{proof}
We first note that if agent $\ell\in[N]$ is leftmost at time $t$, then every signal
the agent is receiving at this time comes from the right.
Indeed, since $\ell$ is leftmost, we have $x_j(t)\geq x_\ell(t)$ for every $j\in[N]$.
If, for some $j\ne \ell$, we had
$x_j(r_{\ell j}(t))<x_\ell(t)$, then \eqref{def:tauij} would give
\[
    x_\ell(t)-x_j(r_{\ell j}(t))    =\cvel\tau_{\ell j}(t).
\]
Consequently,
\[
    x_j(t)-x_j(r_{\ell j}(t)) =    \bigl(x_j(t)-x_\ell(t)\bigr)    +    \bigl(x_\ell(t)-x_j(r_{\ell j}(t))\bigr)  \geq \cvel\tau_{\ell j}(t),
\]
contradicting
\[
    |x_j(t)-x_j(r_{\ell j}(t))|    \leq\svel\tau_{\ell j}(t)
\]
with $\svel<\cvel$. Hence $x_j(r_{\ell j}(t))\geq x_\ell(t)$ for every $j\ne \ell$ and $t\geq t_0$.
Then the light-cone relation \eqref{def:tauij} reads
\[
   \cvel \tau_{\ell j}(t) = x_j(r_{\ell j}(t)) - x_\ell(t),
\]
and differentiation at time $t>t_0$ gives
\[
    (\cvel+v_j(r_{\ell j}(t))) \dot\tau_{\ell j}(t) =  v_j(r_{\ell j}(t))-v_\ell(t).
\]
It follows that
\[
    \frac{\d}{\d t}    \Psi_{\cvel}(\tau_{\ell j}(t))
    = \psi (\cvel \tau_{\ell j}(t)) \dot \tau_{\ell j}(t)
    =    \psi_{\ell j}(t)    \frac{v_j(r_{\ell j}(t))-v_\ell(t)}{\cvel+v_j(r_{\ell j}(t))}.
\]
Using \eqref{eq:CS2} and \eqref{def:FB}, we obtain for $t>t_0$,
\begin{align*}
    \frac{\d}{\d t}F_\ell(t;t_0)    
    &=  \sum_{j\ne \ell}    \psi_{\ell j}(t)\frac{v_j(r_{\ell j}(t))-v_\ell(t)}{\cvel+v_j(r_{\ell j}(t))} - \frac{\dot v_\ell(t)}{\cvel+v_\ell(t)}  \\
    &=\sum_{j\ne \ell}    \psi_{\ell j}(t)    \bigl(v_j(r_{\ell j}(t))-v_\ell(t)\bigr)
    \left(       \frac1{\cvel+v_j(r_{\ell j}(t))}- \frac{1}{\cvel+v_\ell(t)}     \right)
    \\
    &=
    -\sum_{j\ne \ell}    \psi_{\ell j}(t)    \frac{|v_j(r_{\ell j}(t))-v_\ell(t)|^2}
    {(\cvel+v_\ell(t))(\cvel+v_j(r_{\ell j}(t)))}
    \\ & \leq 0.
\end{align*}
Consequently,  $F_\ell(t;t_0)\leq F_\ell(t_0;t_0)$ for all $t\geq t_0$,
and recalling the uniform bound $|v_\ell|\leq\svel$, we have
\begin{align*}
    \Psi_{\cvel}\bigl(\tau_{\ell j}(t)\bigr)   
     \leq  F_\ell(t;t_0) + \ln\left(\frac{\cvel+\svel}{\cvel-\svel}\right)
    \leq  F_\ell(t_0;t_0) + \ln\left(\frac{\cvel+\svel}{\cvel-\svel}\right)
    = \sum_{j\neq \ell} \Psi_{\cvel}(\tau_{\ell j}(t_0)) + \ln\left(\frac{\cvel+\svel}{\cvel -\svel} \right) .
\end{align*}
The fat-tail condition \eqref{cond:flocking} then implies
\begin{equation*} 
    \tau_{\ell j}(t)\leq  \overline\tau    \qquad  \mbox{for all } j\neq\ell,\quad t\geq t_0,
\end{equation*}
with $\overline\tau$ given by \eqref{def:otau1}.
Since $\ell$ is a leftmost agent,
\[
    \dx(t)  =    \max_{j\in[N]}|x_\ell(t)-x_j(t)|.
\]
The lower estimate in \eqref{eq:delay-distance} gives $|x_\ell(t)-x_j(t)|    \leq    (\cvel+\svel)\tau_{\ell j}(t)$,
so that
\[
    \dx(t)\leq(\cvel+\svel) \overline\tau    \qquad \mbox{for all } t\geq t_0.
\]
Finally, for arbitrary $i\neq j$, the upper estimate in
\eqref{eq:delay-distance} yields
\[
    \tau_{ij}(t)    \leq    \frac{|x_i(t)-x_j(t)|}{\cvel-\svel}
    \leq    \frac{\dx(t)}{\cvel-\svel}
    \leq    \frac{\cvel+\svel}{\cvel-\svel} \overline\tau 
   \qquad \mbox{for all } t\geq t_0.
\]
\end{proof}

We hypothesise that any solution of \eqref{eq:CS1}--\eqref{cond:sublum} only exhibits a finite number of switches,
i.e., that the situation treated in Lemma \ref{lem:eventual-leftmost-delay-bound} above is generic.
However, we are not able to prove this hypothesis.
Therefore, we now consider the case when there are infinitely many genuine switching times $(t_n)_{n\in\N}$ in $(t_0,\infty)$,
accumulating only at infinity.
Order the switching times chronologically and denote $\ell_n$ the index of a leftmost agent on the time interval $(t_n,t_{n+1})$, i.e.,
\[   
    x_{\ell_n}(t)=\min_{i\in[N]}x_i(t),    \qquad t_n<t<t_{n+1}.
\]

\begin{lemma}\label{lem:leftmost-concatenation}
For any $m\in\N$ we have
\[
   \sum_{j\neq \ell_m} \Psi_{\cvel}(\tau_{\ell_mj}(t_m))\leq  \sum_{j\neq \ell_0}  \Psi_{\cvel}(\tau_{\ell_0j}(t_0)) + \sum_{n=0}^{m-1} \ln\left( \frac{\cvel + v_{\ell_n}(t_{n+1})}{\cvel + v_{\ell_n}(t_n)} \right) .
\]
\end{lemma}
\begin{proof}
Choose $1\leq n \leq m$. Since $x_{\ell_{n}}(t_{n})=x_{\ell_{n-1}}(t_{n})$, the light-cone relation \eqref{def:tauij} gives
\[
   \tau_{\ell_{n}\ell_{n-1}}(t_{n})  =    \tau_{\ell_{n-1}\ell_{n}}(t_{n})  =0 \qquad\mbox{and}\qquad
    \tau_{\ell_{n}j}(t_{n})    =    \tau_{\ell_{n-1}j}(t_{n}) \mbox{ for all } j\notin\{\ell_{n},\ell_{n-1}\}.
\]
Repeating the same steps as in the proof of Lemma \ref{lem:eventual-leftmost-delay-bound}
for $F_{\ell_n}(t;t_n)$ restricted to $t\in (t_{n},t_{n+1})$ gives
$\frac{\d}{\d t}F_{\ell_n}(t;t_n) \leq 0$ for $t\in (t_{n},t_{n+1})$. Consequently,
\begin{align*}
  \sum_{j\neq \ell_n} \Psi_{\cvel}(\tau_{\ell_nj}(t_n))
  &= \sum_{j\neq \ell_{n-1}} \Psi_{\cvel}(\tau_{\ell_{n-1}j}(t_n))\\
  &= F_{\ell_{n-1}}(t_n;t_{n-1}) + \ln\left(\frac{\cvel + v_{\ell_{n-1}}(t_n)}{\cvel + v_{\ell_{n-1}}(t_{n-1})}\right) \\
  &\leq F_{\ell_{n-1}}(t_{n-1};t_{n-1}) + \ln\left(\frac{\cvel + v_{\ell_{n-1}}(t_n)}{\cvel + v_{\ell_{n-1}}(t_{n-1})}\right) \\
  &= \sum_{j\neq \ell_{n-1}}  \Psi_{\cvel}(\tau_{\ell_{n-1}j}(t_{n-1})) + \ln\left(\frac{\cvel + v_{\ell_{n-1}}(t_n)}{\cvel + v_{\ell_{n-1}}(t_{n-1})}\right),
\end{align*}
and summing over $n=1,\ldots,m$ gives the result.
\end{proof}

Under the fat-tail condition \eqref{cond:flocking}, Lemma \ref{lem:leftmost-concatenation} yields a uniform bound
on the delays and spatial diameter provided that
\(   \label{eq:Cbound}
  \sup_{m\in\N} \,\sum_{n=0}^{m} \ln\left( \frac{\cvel + v_{\ell_n}(t_{n+1})}{\cvel + v_{\ell_n}(t_n)} \right) <\infty.
\)
However, we were not able to prove that this condition universally holds.

\section{Outlook: Multiple agents in multiple dimensions}
\label{sec:optical-action}
Finally, we record a multidimensional analogue of the one-dimensional
functional $F_i$ defined in \eqref{def:FB}.
The resulting functional satisfies a dissipation identity,
however, we were not able to establish its
boundedness from below. Consequently, the calculations carried out
in this section do not imply a uniform bound on the delays and the spatial diameter.
We consider them a motivation for future work.

Fix $i\in[N]$, and let $I_i$ be an interval on which
$\tau_{ij}(t)>0$ for every $j\neq i$. Define
\begin{equation*}  \label{def:optical-los}
    \omega_{ij}(t)
    :=
    \frac{x_j(r_{ij}(t))-x_i(t)}
    {\cvel\tau_{ij}(t)}
    \in\mathbb S^{d-1}.
\end{equation*}
Thus, $\omega_{ij}$ is the line-of-sight direction pointing from the
receiver $i$ towards the retarded position of the emitter $j$. We
further introduce
\begin{equation*}  \label{def:optical-pqg}
    p_{ij}
    :=
    \cvel+\omega_{ij}\cdot v_j(r_{ij}),
    \qquad
    q_{ij}
    :=
    \cvel+\omega_{ij}\cdot v_i.
\end{equation*}
The subluminal velocity bound gives
\begin{equation*}  \label{eq:optical-positive-denominators}
    \cvel-\svel
    \leq p_{ij}(t),q_{ij}(t)
    \leq \cvel+\svel.
\end{equation*}

Differentiating the light-cone relation \eqref{def:tauij}
and taking the scalar product with $\omega_{ij}$, we obtain
\begin{equation}\label{eq:optical-delay-kinematics}
    p_{ij}\dot\tau_{ij}=p_{ij}-q_{ij},
    \qquad
    \dot r_{ij}
    =
    1-\dot\tau_{ij}
    =
    \frac{q_{ij}}{p_{ij}}.
\end{equation}
The quotient $q_{ij}/p_{ij}$ is the Doppler ratio associated
with the corresponding retarded line of sight. Given a reference time $t_0\in I_i$, define the functional
\(  \label{def:Gi}
   G_i(t;t_0):= \sum_{j\neq i} \left( \Psi_{\cvel}\bigl(\tau_{ij}(t)\bigr) - \int_{t_0}^t    \psi_{ij}(s)    \frac{p_{ij}(s)-q_{ij}(s)}{q_{ij}(s)}    \,\d s\right),
\)
where we recall that $\psi_{ij}(s) = \psi(\cvel\tau_{ij}(s))$.

\begin{lemma}
\label{lem:optical-action-dissipation}
Fix $i\in[N]$, and let $I_i$ be an interval on which $\tau_{ij}(t)>0$ for every $j\neq i$.
Then the functional $G_i$ defined in \eqref{def:Gi} satisfies, for all $t\in I_i$,
\begin{equation}    \label{eq:optical-action-dissipation}
    \frac{\d}{\d t}G_i(t;t_0) = - \sum_{j\neq i}    \psi_{ij}(t)
       \frac{ \bigl|    \omega_{ij}(t)\cdot    \bigl(v_j(r_{ij}(t))-v_i(t)\bigr) \bigr|^2}{\bigl(\cvel+\omega_{ij}(t)\cdot v_j(r_{ij}(t))\bigr)\bigl(\cvel+\omega_{ij}(t)\cdot v_i(t)\bigr)}
    \leq0.
\end{equation}
\end{lemma}

\begin{proof}
Using \eqref{eq:optical-delay-kinematics} in $\frac{\d}{\d t} \Psi_{\cvel}(\tau_{ij}) = \psi (\cvel \tau_{ij}) \dot \tau_{ij}$, we have
\[
    \frac{\d}{\d t}G_i(t;t_0)   =     \sum_{j\neq i}  \psi_{ij}(p_{ij}-q_{ij})    \left(  \frac1{p_{ij}}-\frac1{q_{ij}} \right)
   =  -  \sum_{j\neq i}    \psi_{ij}    \frac{(p_{ij}-q_{ij})^2}{p_{ij}q_{ij}},
\]
which is \eqref{eq:optical-action-dissipation}.
\end{proof}

The problem is that, in several dimensions, no uniform upper bound
on 
\( \label{eq:exA}
   \sum_{j\neq i} \int_{t_0}^t    \psi_{ij}(s)    \frac{p_{ij}(s)-q_{ij}(s)}{q_{ij}(s)}    \,\d s
 \)
is presently available. Therefore, \eqref{eq:optical-action-dissipation}
gives neither a uniform bound on $\Psi_{\cvel}\bigl(\tau_{ij}(t)\bigr)$, and hence on the delays,
nor an a priori bound on the total accumulated dissipation.

We only note that for $d=1$, the functional $G_i$ reduces to $F_i$
introduced in \eqref{def:FB} to study the one-dimensional setting.
Indeed, let agent $i$ be leftmost throughout the
interval under consideration. As shown above, every signal received by
$i$ then comes from the right, so that
\[
    \omega_{ij}=1,    \qquad    q_{ij}=\cvel+v_i.
\]
Consequently, \eqref{eq:CS2} gives
\[
      \dot v_i =  \sum_{j\neq i}  \psi_{ij}(p_{ij}-q_{ij}) .
\]
It follows that
\[
    \sum_{j\neq i} \int_{t_0}^t    \psi_{ij}(s)    \frac{p_{ij}(s)-q_{ij}(s)}{q_{ij}(s)}  \,\d s = 
     \int_{t_0}^t    \frac{\dot v_i(s)}    {\cvel+v_i(s)}  \,\d s
      = \ln\left(   \frac{\cvel+v_i(t)}{\cvel+v_i(t_0)}    \right),
\]
and, therefore, $G_i(t;t_0)    =  F_i(t;t_0).$
Obviously, the problem of proving the bound \eqref{eq:Cbound} is related to the problem of bounding \eqref{eq:exA} from above.

\bibliographystyle{amsplain}

\end{document}